\documentclass[reqno,10pt]{amsart}
\usepackage{euscript}
\usepackage{amssymb}
\usepackage{amscd}
\usepackage{amsmath}
\usepackage{graphics}
\usepackage{epsfig}
\usepackage{color}
\usepackage[all]{xy}

\numberwithin{equation}{section}
\newtheoremstyle{my}{1.5em}{0.5em}{\em}{}{\sc}{.}{0.5em}{}
\newtheoremstyle{your}{1.5em}{0.5em}{}{}{\sc}{.}{0.5em}{}
\theoremstyle{my}

\theoremstyle{my}
\newtheorem{thm}{Theorem}[section]
\newtheorem{Theorem}[thm]{Theorem}
\newtheorem*{Theorem*}{Theorem}
\newtheorem{Corollary}[thm]{Corollary}

\newtheorem*{corollary*}{Corollary}
\newtheorem{Lemma}[thm]{Lemma}

\newtheorem{Proposition}[thm]{Proposition}

\newtheorem*{conjecture*}{Conjecture}

\newtheorem*{question*}{Question}

\newtheorem*{definitions*}{Definitions}

\newtheorem*{rem*}{Remark}

\newtheorem*{remark*}{Remark}

\newtheorem*{remarks*}{Remarks}
\newtheorem*{example*}{Example}
\newtheorem*{examples*}{Examples}

\newtheorem*{convention*}{Convention}
\newtheorem*{conventions*}{Conventions}
\newtheorem*{Note*}{Note}
\newtheorem*{exercise*}{Exercise}
\newtheorem*{bibliographical-note*}{Bibliographical note}

\theoremstyle{your}
\newtheorem{Remark}[thm]{Remark}

\newcommand{\R}{\mathbb{R}}
\newcommand{\Z}{\mathbb{Z}}

\newcommand{\C}{\mathbb{C}}
\newcommand{\ev}{\operatorname{ev}}

\newcommand{\pa}{\partial}

\newcommand{\bR}{\mathbb{R}}
\newcommand{\bZ}{\mathbb{Z}}

\newcommand{\bC}{\mathbb{C}}

\newcommand{\bP}{\mathbb{P}}

\renewcommand{\ker}{\mathrm{ker}}

\newcommand{\scrF}{\EuScript{F}}
\newcommand{\scrW}{\EuScript{W}}

\newcommand{\MM}{\mathcal{M}}
\newcommand{\FF}{\mathcal{F}}
\newcommand{\BB}{\mathcal{B}}

\newcommand{\XX}{\mathcal{X}}

\newcommand{\NN}{\mathcal{N}}

\newcommand{\TT}{\mathcal{T}}

\newcommand{\DD}{\mathcal{D}}

\newcommand{\st}{\mathrm{st}}
\newcommand{\pl}{\mathrm{pl}}

\newcommand{\Pre}{\operatorname{PG}}

\title{Lagrangian exotic spheres}
\author{Tobias Ekholm}
\address{Department of mathematics Uppsala University, Box 480, 751 06 Uppsala, Sweden \newline
Institut Mittag-Leffler, Aurav. 17, 182 60 Djursholm, Sweden}
\email{tobias.ekholm\@@math.uu.se}
\author{Thomas Kragh}
\address{Department of mathematics Uppsala University, Box 480, 751 06 Uppsala, Sweden \newline
Institut Mittag-Leffler, Aurav. 17, 182 60 Djursholm, Sweden}
\email{thomas.kragh\@@math.uu.se}
\author{Ivan Smith}
\address{Centre for Mathematical Sciences, University of Cambridge, \newline Wilberforce Road, CB3 0WB, United Kingdom}
\email{is200\@@cam.ac.uk} 
\date{8th March 2015. v2: 17th September 2015.}

\thanks{TE is partially supported by the Knut and Alice Wallenberg Foundation as a Wallenberg Scholar and by the Swedish Research Council, 2012-2365.}
\begin{document}
\thispagestyle{empty}
\maketitle

\begin{abstract} 
Let $k>2$. We prove that the cotangent bundles $T^*\Sigma$ and $T^*\Sigma'$ of oriented homotopy $(2k-1)$-spheres  $\Sigma$ and $\Sigma'$ are symplectomorphic only if $[\Sigma] =  [\pm\Sigma'] \in \Theta_{2k-1} / bP_{2k}$, where $\Theta_{2k-1}$ denotes the group of oriented homotopy $(2k-1)$-spheres under connected sum, $bP_{2k}$ denotes the subgroup of those that bound a parallelizable $2k$-manifold, and where $-\Sigma$ denotes $\Sigma$ with orientation reversed. We further show  that if $n=4k-1$ and $\bR\bP^n \#\Sigma$ admits a Lagrangian embedding in $\bC\bP^n$, then $[\Sigma\#\Sigma] \in bP_{4k}$. 
The proofs build on \cite{Abouzaid} and \cite{EkholmSmith} in combination with a new cut-and-paste argument; that also
yields some interesting explicit exact Lagrangian embeddings, for instance of the sphere $S^n$ into the plumbing $T^*\Sigma^n \#_{\mathrm{pl}} T^*\Sigma^n$ of cotangent bundles of certain exotic spheres. As another application, we show that there are re-parameterizations of the zero-section in the cotangent bundle of a sphere that are not Hamiltonian isotopic (as maps rather than as submanifolds) to the original zero-section.  
\end{abstract}

\maketitle

\section{Introduction}
Arnold's ``nearby Lagrangian conjecture'' asserts that any exact Lagrangian submanifold of the cotangent bundle $T^{\ast} M$ of a closed smooth manifold $M$, with its standard symplectic structure, is Hamiltonian isotopic to the $0$-section. It would imply the weaker statement that if a closed smooth manifold $N$ admits an exact Lagrangian embedding into $T^{\ast}M$ then $M$ and $N$ are diffeomorphic;  it is known \cite{FSS, Nadler, Abouzaid:htpy, Kragh} under these hypotheses that $M$ and $N$ are homotopy equivalent. The weaker statement furthermore  implies that if the cotangent bundle $T^{\ast}N$ is symplectomorphic to $T^{\ast}M$ then $N$ is diffeomorphic to $M$. (To see this, consider the image of the $0$-section in $N$ under a symplectomorphism to $T^{\ast}M$.)   
In light of these results and conjectures, it is natural to ask to what extent the symplectic topology of $T^*M$ remembers the smooth structure on $M$.

Let $\Theta_n$ denote the group of oriented homotopy $n$-spheres, where $n\geq 5$.  The group operation is connect sum, and the inverse of an element is given by reversing its orientation.  There is a subgroup $bP_{n+1} \subset \Theta_n$ of homotopy spheres which bound parallelizable manifolds, see \cite{KM}.

\begin{Theorem} \label{Thm:One}
Let $n>4$ be an odd integer. If $\Sigma$ and $\Sigma'$ belong to $\Theta_n$ and $T^*\Sigma$ is symplectomorphic to $T^*\Sigma'$, then $[\Sigma] = \pm [\Sigma'] \in \Theta_{n}/bP_{n+1}$.
\end{Theorem}

In other words, perhaps after reversing orientation, $\Sigma$ and $\Sigma'$ agree in the quotient group $\Theta_n / bP_{n+1}$. Theorem \ref{Thm:One} extends the following groundbreaking result of Abouzaid \cite{Abouzaid}: if $n=4k+1$, if $\Sigma \in \Theta_n$, and if $T^*\Sigma$ is symplectomorphic to $T^*S^n$ then $[\Sigma] \in bP_{n+1}$. Abouzaid constructed a parallelizable filling of $\Sigma$ from moduli spaces of holomorphic disks.  Using similar techniques, \cite{EkholmSmith} showed that if $n=8k$, if $\Sigma \in \Theta_{n}$, and if $T^*(S^1 \times S^{8k-1})$ and $T^*((S^1 \times S^{8k-1})\# \Sigma)$ are symplectomorphic, then $\Sigma \in bP_{n+1}$. 

The proof of Theorem \ref{Thm:One} has two main ingredients. First, we introduce a cut-and-paste argument that allows one to pass from distinguishing the cotangent bundle of an exotic sphere from that of the standard sphere to distinguishing cotangent bundles of two exotic spheres. Together with Abouzaid's result that establishes Theorem \ref{Thm:One} in dimensions $n=4k+1$. Second, a modification 
of the argument of \cite{EkholmSmith} allows us to study cotangent bundles of products $S^1 \times \Sigma$, where $\Sigma \in \Theta_n$ is a homotopy $n$-sphere, rather than of connect sums $(S^1 \times S^{n}) \# \Sigma'$ for $\Sigma' \in \Theta_{n+1}$. That modification, combined with results on regular homotopy classes of exact Lagrangian embeddings in cotangent bundles \cite{AbKr} for $n=8k+3$, then covers the dimensions $n=4k-1$, see Theorem \ref{thm:4k-1}. The results on $S^1\times \Sigma$ can be transported to a result on Lagrangian homotopy real projective spaces in complex projective space $\bC\bP^n$: the boundary of a tubular neighborhood of $\C \bP^{n}\subset \C \bP^{n+1}$ restricted to such a Lagrangian in $\C \bP^{n}$ gives a Lagrangian in $\bC\bP^{n+1}\backslash \bC\bP^n\approx B^{2n+2} \subset \C^{n+1}$ (``Biran's circle bundle construction"), to which the previous results can be applied. 

\begin{Theorem} \label{Thm:Two}
Let $n=4k-1$ and $\Sigma \in \Theta_{4k-1}$ be an oriented homotopy sphere.  If $\bR\bP^n\#\Sigma$ admits a Lagrangian embedding in $\bC\bP^n$, then $[\Sigma \# \Sigma] \in bP_{4k}$.
\end{Theorem}

The constraints of Theorems \ref{Thm:One} and \ref{Thm:Two} arise from moduli spaces of holomorphic discs and from  Floer theory, and the underlying arguments can sometimes be implemented for exact Lagrangian submanifolds of Weinstein manifolds that are not necessarily cotangent bundles. For instance, \cite{AbSm} showed that in the plumbing $T^*S^n \,\#_{\mathrm{pl}}\, T^*S^n$ of two copies of $T^*S^n$, every exact Lagrangian submanifold of vanishing Maslov class is a homotopy sphere. This plumbing is symplectomorphic to the Milnor fibre $A^n_2 =   \left\{\sum_{i=1}^n x_i^2 + z^{3} = 1\right\} \subset (\bC^{n+1}, \omega_{\st})$ 
equipped with the restriction of the flat K\"ahler form from Euclidean space. 
In contrast with Theorem \ref{Thm:One},  the aforementioned cut-and-paste argument shows the following.

\begin{Proposition}
Every multiple of $\Sigma$ in $\Theta_n$ admits a Lagrangian embedding into 
$T^*\Sigma \,\#_{\mathrm{pl}}\, T^*S^n$. 
\end{Proposition}

This note contains several other results which are variations on this theme (and which are fairly straightforwad consequences of \cite{Abouzaid, EkholmSmith}, in combination with more classical Floer-theoretic arguments): embedding or non-embedding results for exotic spheres in certain plumbings or boundary connect sums of cotangent bundles,  constraints on contactomorphisms of unit cotangent bundles, and existence of exotic Lagrangian concordances of certain Legendrian spheres.  

To amplify the last of these, from the perspective of Legendrian surgery, Theorem \ref{Thm:One} shows that the re-parameterization of certain Legendrian spheres may change their Legendrian isotopy class, and that leads to a similar result for cotangent bundles. To state the result, we recall that there is a natural identification between the path components of the group of orientation-preserving diffeomorphisms $\pi_{0}(\mathrm{Diff}(S^n))$ and the group $\Theta_{n+1}$ of oriented homotopy spheres, via the clutching function construction.
\begin{Theorem}\label{t:nearbyparam}
Let $n\ge 6$ be even and let $\phi\colon S^{n}\to S^{n}$ be a diffeomorphism defining a homotopy sphere $\Sigma_{\phi} \in \Theta_{n+1} \backslash bP_{n+2}$.  Let $\iota\colon S^{n}\to T^{\ast}S^{n}$ denote the inclusion of the $0$-section. Then the maps $\iota$ and $\iota\circ\phi$ are not Hamiltonian isotopic. 
\end{Theorem}

%new
This result was independently (and originally) obtained by Evans and Dmitroglou Rizell, and appears as Corollary 5.1 in the first arXiv version of \cite{DRizellEvans}.
%endnew
We point out that $\iota$ and $\iota\circ\phi$ in Theorem \ref{t:nearbyparam} are smoothly isotopic by \cite{Haefliger}.
% Theorem \ref{t:nearbyparam} can be viewed as the failure of a ``parametrized" version of the nearby Lagrangian conjecture, concerning maps rather than submanifolds. 
%new
There is a folklore strengthened version of the nearby Lagrangian conjecture, which states that for a closed manifold $M$,  the space of embedded closed exact Lagrangians in $T^*M$ is contractible. This is equivalent to the statement that the map from the space of diffeomorphisms of $M$ to the space of parametrized Lagrangian embeddings $M \to T^*M$ is a homotopy equivalence. Theorem \ref{t:nearbyparam} is consistent with this stronger conjecture.
%endnew

\section{Nearby Lagrangian homotopy $(4k-1)$-spheres}
The main result in \cite{Abouzaid} says that, for a homotopy sphere $\Sigma$ of dimension $4k+1$, if $T^{\ast}\Sigma$ is symplectomorphic to $T^{\ast} S^{4k+1}$ then $[\Sigma]\in bP_{4k+2}$ (since $bP_{n+1} \leq \Theta_n$ is a subgroup the conclusion is independent of the choice of orientation on $\Sigma$). In this section we describe an extension of the results in \cite{EkholmSmith} which, together with results on Lagrangian regular homotopy classes of exact Lagrangian embeddings in cotangent bundles from \cite{AbKr}, gives the corresponding result in dimensions $4k-1$. A more precise statement of the main result of this section is as follows.

\begin{Theorem}\label{thm:4k-1}
Let $\Sigma$ be a homotopy $(4k-1)$-sphere. If $\Sigma$ admits a Lagrangian embedding into $T^{\ast}S^{4k-1}$ with its standard symplectic form then $[\Sigma]\in bP_{4k}$. In particular,  if $[\Sigma]\notin bP_{4k}$ then the cotangent bundles $T^{\ast}\Sigma$ and $T^{\ast}S^{4k-1}$ are not symplectomorphic.
\end{Theorem}

\subsection{Lagrange surgery on the Whitney sphere}
Let $\pi\colon \C^{2k} \to \C$ be the model Lefschetz fibration 
\[
\pi(z_1,\ldots, z_{2k})= \sum_{j=1}^{2k} z_j^2.
\]
Let $\gamma$ be a smooth embedded closed curve in $\C$ passing through the origin.  The union of the vanishing cycles of $\pi$ along $\gamma$ defines an immersed Lagrangian sphere $K\subset\C^{2k}$ with a single transverse double point, which is one model for the Whitney immersion.  For instance, if $B_{0}(\delta)\subset\C$ denotes the disk of radius $\delta$ centered at the origin and if $\gamma \cap B_0(\delta) = (-\delta, \delta) \subset \R$, then the vanishing cycles $V_t = \sqrt{t}S^{2k-1} \subset \pi^{-1}(t)$ along $\gamma$ are locally given by
\[
\bigcup_{t \in [0,\delta)} V_t \ = \ \R^{2k}; \quad \bigcup_{t\in (-\delta,0]} V_t = i\R^{2k}
\]
which meet transversely. There are two Lagrange surgeries $K_+$ and $K_-$ of $K$ which are given by perturbing $\gamma$ to an embedded curve in $\C-\{0\}$ which respectively does or  does not enclose the origin. The surgeries have Maslov number $\mu(K_-)=2$ and $\mu(K_+)=2k$, see e.g.~\cite[Lemma 2.5]{EkholmSmith}. The Maslov $2k$ surgery $K_+$ is Lagrangian isotopic to the Lagrangian $S^{1}\times S^{2k-1}$ in the model Lefschetz fibration,
\[
K_+ \ = \ \bigcup_{t\in S^1} \sqrt{t} S^{2k-1},
\]
which is also the mapping torus of the antipodal map on $S^{2k-1}$. In what follows we will denote this Maslov $2k$ embedding $w\colon S^{1}\times S^{2k-1}\to\C^{2k}$. By micro-extension of Lagrangian embeddings we find that $w$ extends to a symplectic embedding 
\begin{equation}\label{eq:microext}
W\colon \mathcal{U}\to \C^{2k},
\end{equation}
where $\mathcal{U}\subset T^{\ast} (S^{1}\times S^{2k-1})$ is a neighborhood of the $0$-section. 
 
Let $\Sigma$ be a $(4k-1)$-dimensional homotopy sphere and assume that $\Sigma$ admits a (necessarily  exact) Lagrangian embedding $\iota'\colon\Sigma\to T^{\ast} S^{4k-1}$. Multiplying by $S^{1}$ we find that  $S^{1}\times\Sigma$ admits an exact Lagrangian embedding $\iota\colon S^{1}\times\Sigma\to T^{\ast} (S^{1}\times S^{4k-1})$. After fiber-scaling we may assume that the image of $\iota$ lies inside $\mathcal{U}$, see \eqref{eq:microext}. Consider the composite Lagrangian embedding 
\begin{equation}\label{eq:composition}
f:=W\circ\iota\colon S^{1}\times\Sigma\to \C^{4k}
\end{equation}
with image $L=f(S^1\times \Sigma)$. Any such Lagrangian embedding determines a stable Lagrangian Gauss map
\begin{equation} \label{Eqn:Gauss}
G_f\colon L \to U(4k)/O(4k) \to U/O,
\end{equation}
given at a point $q\in L$ by its tangent space $T_qL \subset \C^{4k}$ considered as a stable Lagrangian subspace in $U/O$.

Consider the space $\mathcal{X}$ of maps of the 2-disk $D$ into $\C^{4k}$ with boundary in $L$ and for $u\in\mathcal{X}$ consider the $\bar{\pa}$-operator acting on vector fields $v\colon D\to\C^{4k}$ along $u$ that are tangent to $L$ along the boundary, i.e.~for $z\in\partial D$, $v(z)\in T_{u(z)}L$. Associated to this data is an index bundle over $\mathcal{X}$ that controls the tangent bundle of the space of (perturbed) holomorphic disks with boundary on $L$. As in \cite{EkholmSmith} we claim that in the present case this index bundle is completely determined by the corresponding index bundle for $K_+$.  More precisely, we have the following result.

\begin{Lemma} \label{Lem:GaussInterpolate}
There is a one-parameter family of $C^0$-homeomorphisms $\phi_t\colon \C^{4k} \to \C^{4k}$ and a 1-parameter family of continuous maps $\psi_t\colon L\approx_{C^{0}} S^1 \times S^{4k-1}  \to U/O$, with the properties:
\begin{enumerate}
\item $\phi_0$ is the identity;
\item $\phi_1(w(S^{1}\times S^{4k-1})) = f(S^{1}\times\Sigma)$;
\item $\psi_0=G_w$ and $\psi_1 = G_{f} \circ \phi_1$.
\end{enumerate}
\end{Lemma}

\begin{proof}
Any homotopy $n$-sphere $\Sigma$, $n>4$, admits a Morse function with exactly two critical points \cite{KM}, and hence admits the structure of a {\sc pl}-manifold {\sc pl}-homeomorphic to the standard sphere $S^n$. Thus, $\Sigma\approx_{\text{\sc pl}} S^{4k-1}$ and we will identify the two as {\sc pl}-manifolds.

The composition of $\iota'$ and the projection to the $0$-section has degree $\pm 1$, by \cite{Seidel:kronecker}, so the homotopy groups of the mapping cone on $\iota'$ relative $T^{\ast} S^{4k-1}$ all vanish. 
Moreover, the codimension of the {\sc pl}-embedding $\iota' \colon \Sigma\to T^{\ast} S^{4k-1}$ is $4k-1>3$, and so it follows from \cite{Hudson} that $\iota'$ is ambient {\sc pl}-isotopic to the inclusion of the $0$-section. Denote the ambient {\sc pl}-isotopy $\theta_{t}$, $t\in[0,1]$. 

Note next that the tangent bundle of $T^{\ast} S^{4k-1}$ equipped with a complex structure compatible with the standard symplectic form is trivial as a complex vector bundle since $\pi_{4k-2}(U(4k-1))=0$. Thus the immersion $\iota$ and the inclusion of the $0$-section both determine stable Lagrangian Gauss maps $G_{\iota}$ and $G_{0}$ into $U/O$. We next show that there is a homotopy of Gauss maps covering the {\sc pl}-isotopy between $\iota$ and the $0$-section, or more simply stated that $G_{\iota}$ and $G_{0}$ are homotopic. 

Consider first the case when $4k-1=8j-1$. By Bott-periodicity $\pi_{8j-1}(U/O)=0$ and the required homotopy exists. In the case $4k-1=8j+3$, Bott-periodicity gives $\pi_{8j+3}(U/O)=\Z_{2}$ but \cite[Equation (2.16) \& Table 1, column labeled $k\;\mathrm{mod}\,8=2$]{AbKr} implies that the homotopy class of the Gauss map of any exact Lagrangian embedding equals $0\in\Z_{2}$. We conclude that $G_{\iota}$ and $G_{0}$ are again homotopic  in this case. Pick a homotopy of Gauss maps $\rho_{t}$, $t\in[0,1]$, covering $\theta_{t}$. 
  
Let $\sigma\colon S^{1}\to T^{\ast} S^{1}$ be the inclusion of the $0$-section and $\tau$ denote its tangent line in $T(T^{\ast} S^{1})$. The lemma is proved by taking the maps 
\begin{align*}
&\sigma\times\theta_{t}\colon S^{1}\times S^{4k-1}\to T^{\ast}S^{1}\times T^{\ast}S^{4k-1}=T^{\ast}(S^{1}\times S^{4k-1}),\\
&\tau\times \rho_{t}\colon S^{1}\times S^{4k-1}\to U/O,
\end{align*}
and then composing with $W\colon \mathcal{U} \to\C^{4k}$, i.e., $\phi_{t}=W(\sigma\times\theta_{t})$ and $\psi_{t}=dW(\tau\times\rho_{t})$.
\end{proof}

\subsection{Bounding manifolds from Hamiltonian displacements}
In this section we give a brief description of the main construction in \cite{EkholmSmith}.
With notation as in \eqref{eq:composition}, we let $L=f(S^{1}\times\Sigma)\subset \C^{4k}$. It follows from Lemma \ref{Lem:GaussInterpolate} that the minimal Maslov number of $L$ is $4k$. By approximation, see \cite[Lemma 4.1]{EkholmSmith}, we may furthermore assume that $L$ is real analytic.

Let $H=H_{t}\colon \C^{4k}\to \C^{4k}$ be a time dependent Hamiltonian function so that the time 1 flow of its Hamiltonian vector field $X_{H}$ displaces $L$ from itself. As in \cite[Section 3.1, Equation (3.3)]{EkholmSmith} we define a 1-parameter family of 1-forms $\gamma_{r}$, $r\in[0,\infty)$ on the 2-disk $D$ which vanishes in a neighborhood of the boundary and which equals $0$ for $r=0$. We then consider the Floer equation for maps $u\colon (D,\partial D)\to (\C^{4k},L)$,
\begin{equation}\label{eq:Floer}
(du + \gamma_{r}\otimes X_{H})^{0,1}=0,
\end{equation} 
where the complex anti-linear part is taken with respect to some almost complex structure $J$ on $\C^{4k}$ and the standard complex structure on $D$, and where $X_{H}=X_{H}(u(z),t(z))$ for a function $t\colon D\to[0,1]$ that is the pull-back of $t\in [0,1]$ under a conformal equivalence $D\backslash \{\pm 1\} \to\R\times[0,1]$. 

Let $\beta\in H_{1}(L)$ be a generator such that the Maslov index of $\beta$ equals $+4k$.
We write $\FF(j\beta)$, $j\in\Z$ for the space of solutions $u$ of \eqref{eq:Floer} such that $u(\partial D)$ represents the class $j\beta\in H_{1}(L)$, and $\FF^{r}(j\beta)$ for the corresponding space for a fixed parameter $r\in[0,\infty)$. The formal dimensions of these spaces are
\[ 
\dim(\FF(j\beta))=\dim(\FF^{r}(j\beta))+1= 4k(j+1)+1,
\] 
and \cite[Section 4]{EkholmSmith} shows that the solution spaces are transversely cut out $C^{1}$-manifolds for generic data.

For $r=0$, the perturbation term $X_{H}\otimes\gamma_{r}=0$ and \eqref{eq:Floer} reduces to the standard Cauchy-Riemann equation for holomorphic disks in $\C^{4k}$ with boundary on $L$. Since $X_{H}$ is displacing there exists for each $j$ an $r_0>0$ such that \eqref{eq:Floer} does not have any solutions in $\FF^{r}(j\beta)$ for $r>r_0$. In general, we write $\MM(j\beta)$ for the moduli space of unperturbed holomorphic disks in homology class $j\beta$:  
\[ 
\MM(j\beta)=\FF^{0}(j\beta)/G,
\]     
where $G$ is the group of conformal transformations of the disk. These spaces have formal dimensions
\[ 
\MM(j\beta)=4k(j+1)-3.
\] 
In fact, $\MM(j\beta)$ is empty for $j<0$ since the area of any holomorphic disk is non-negative,
whilst \cite[Section 4]{EkholmSmith} uses a gauge fixing procedure to show that for generic almost complex structure, $\MM(\beta)$ is a transversely cut out $(8k-3)$-dimensional $C^{1}$-manifold. 

We next discuss compactifications of these solution spaces. Consider first $\MM(\beta)$. By Gromov compactness this space is compact up to bubbling. However, since $\MM(0\beta)$ consists only of constant maps and $\MM(j\beta)=\varnothing$ for $j<0$, there can be no bubbles, and $\MM(\beta)$ is compact. The solution space $\FF(0\beta)$ has dimension $4k+1$. Its boundary is $\FF^{0}(0\beta)$. By exactness $\FF^{0}(0\beta)=L$ is the space of constant maps to $L$, and it is straightforward to check that this solution space is transversely cut out. Since $\FF^{r}(0\beta)=\varnothing$ for all sufficiently large $r$ it follows that $\FF(0\beta)$ is non-compact and its Gromov-Floer boundary is the space of broken disks with one component in $\FF(-\beta)$ and one in $\MM(\beta)$. Other configurations are ruled out by the  transversality property mentioned above: $\FF(j\beta)=\varnothing$ for $j<-1$. Thus the Gromov-Floer boundary of $\FF(0\beta)$ can be described as a fibered product. More precisely, if $\MM^{\ast}(\beta)$ denotes the moduli space of holomorphic disks with boundary in class $\beta$ and one marked point on the boundary and if $\FF^{\ast}(-\beta)$ denotes the space of disks in $\FF(-\beta)$ with one marked point on the boundary then there are evaluation maps 
\begin{align} \label{eq:evaluation1}
&\ev\colon \MM^{\ast}(\beta)\to L, \notag \\
&\ev\colon \FF^{\ast}(-\beta)\to L, 
\end{align} 
and the Gromov-Floer boundary is
\[ 
\MM^{\ast}(\beta)\times_{L} \FF^{\ast}(-\beta).
\]
 \cite[Section 5]{EkholmSmith} establishes transversality for this fibered product: for generic data 
\[ 
\NN=\MM^{\ast}(\beta)\times_{L} \FF^{\ast}(-\beta)
\]
is a $4k$-dimensional closed $C^{1}$-manifold. 

In \cite[Section 6]{EkholmSmith} a Floer gluing map 
\[ 
\Phi\colon \NN\times [\rho_0,\infty) \to \FF(0\beta)
\]
is constructed and proved to parameterize a neighborhood of the Gromov-Floer boundary (corresponding to $\NN\times\{\infty\}$). We give a brief description of the construction of this map. The starting point is a gauge fixing procedure for maps $v$ in the $\MM^{\ast}(\beta)$-component of elements in the fibered product $\NN$, see \cite[Section 5.2]{EkholmSmith}, which parameterizes the maps by the upper half-plane with the marked point at $\infty$ and with $\pm 1$ mapping to small spheres around the image of the marked point. This parameterization is denoted $v^{\mathrm{st}}$.   

Let $\XX(0\beta)$ denote the configuration space of all smooth maps from the disk into $\C^{4k}$ with boundary in $L$ and in the trivial relative homotopy class.  
If $(u,\zeta)\in \FF^{\ast}(-\beta)=\FF(-\beta)\times\partial D$ and $v\in\MM^{\ast}(\beta)$ with $u(\zeta)=v(1)$,  and if $\rho>0$ is large, then the $\rho$-pregluing $\Pre_{\rho}((u,\zeta),v)$ of $(u,\zeta)$ and $v$ is the element of  $\XX(0\beta)$ defined as follows. It agrees with $u$ outside a half-disk neighborhood of $\zeta$ of radius $2e^{-\rho}$, it agrees with $v^{\mathrm{st}}(e^{2\rho} z)$ for $z$ in the radius $e^{-\rho}$ half-disk neighborhood, and it interpolates between the two maps on the half annular region between the half circles of radii $e^{-\rho}$ and $2e^{-\rho}$, where both maps are at distance $\mathcal{O}(e^{-\rho})$ from the point $u(\zeta)=v(1)$. Applying parameterized Newton iteration to $\Pre_{\rho}$ we obtain solutions $\Phi_{\rho}((u,\zeta),v)$ to the Floer equation. In fact there is a $\rho_0>0$ such that the gluing map $\Phi$ gives a proper $C^{1}$-diffeomorphism from $\Pre([\rho_0,\infty)\times\NN)\subset\XX(0\beta)$ to a neighborhood of the Gromov-Floer boundary of $\FF(0\beta)$. Furthermore, the $C^{1}$-distance between $\Pre([\rho_0,\infty)\times\NN)$ and its image under the gluing map $\Phi$ goes to $0$ as $\rho_0\to\infty$.

\begin{Remark}
The Sobolev spaces used to define the configuration space as a Banach manifold (e.g.~to carry out the Newton iteration)  are not the ones naturally associated to maps from the unit disk. Instead we use a domain which consists of two disks joined by a long strip that grows with the gluing parameter, see \cite[Sections 3.6 and 6.3]{EkholmSmith} for details.
\end{Remark}

Removing a neighborhood of the Gromov-Floer boundary of $\FF(0\beta)$, we find that
\[ 
\overline{\FF}(0\beta)=\FF(0\beta)-\Phi(\NN\times (\rho_0,\infty))
\] 
is a smooth compact submanifold of $\FF(0\beta)$ with boundary 
\[  
\partial\overline{\FF}(0\beta)  \approx_{C^{1}}  
L \cup \NN,
\]
where the diffeomorphism on $L$ is given by inclusion of constant maps and the diffeomorphism on $\NN$ is given by the gluing map $\Phi$, which is arbitrarily close to the pregluing map provided $\rho_0$ is sufficiently large.

Consider the section $s$ of the bundle $\FF^{\ast}(0\beta)\to\FF(0\beta)$ over $\Phi(\NN\times\{\rho_0\})\subset\overline{\FF}(0\beta)$ which takes the solution that corresponds to $\Pre((u,\zeta),v)$ to the point on the boundary $\partial D$ which lies in the middle of the interpolation region in the positive direction from $\zeta$. Evaluation at this point gives a map 
\begin{align} \label{eq:evaluation2}
\ev_s\colon \NN \to L.
\end{align}
Since the gluing map $\Phi$ is close to $\Pre$ which maps the whole interpolation region into a small ball around $u(\zeta)$ it follows that $\ev_s$ is homotopic to the canonical map $\widetilde{\ev}_{s}\colon\NN\to L$ given by the composition 
\begin{equation}\label{fibermap}
\begin{CD}
\NN @>>> \FF^{\ast}(-\beta) @>{\ev_{\zeta}}>> L,
\end{CD} 
\end{equation}
i.e.~the map induced from the fibered product.

Consider the connected components $\mathcal{C}_{j}$, $j=1,\dots, m$ of the solution space $\FF(-\beta)$ and note that $\mathcal{C}_{j}$ is a closed connected 1-manifold, i.e.~a circle. The restriction $\mathcal{C}_{j}^{\ast}$ of the bundle $\FF^{\ast}(-\beta)$ to $\mathcal{C}_{j}$ is thus a torus. Since $H_1(L)\cong \Z$ and since the fiber of $\mathcal{C}_{j}^{\ast}$ maps to a generator of $H_1(L)$ there is a unique way of filling $\mathcal{C}_{j}^{\ast}$ to a solid torus $\mathcal{D}_{j}$ in such a way that the map $\ev_\zeta$ in~\eqref{fibermap} extends to a map $e_j\colon\mathcal{D}_{j}\to L$. We write $\DD=\bigcup_{j=1}^{m}\mathcal{D}_{j}$ for the union of these solid tori and $e\colon \DD\to L$ for the map which agrees with $e_j$ on $\DD_{j}$. Then $\partial\DD=\FF^{\ast}(-\beta)$ and we define, using an extension $e\colon\DD\to L$ in general position with respect to $\ev\colon\MM^{\ast}(\beta)\to L$,
\[ 
\TT =  \MM^{\ast}(\beta)\times_{L} \DD.
\]
Then $\TT$ is a smooth compact manifold with boundary $\partial \TT=\NN$ and 
\[ 
\BB = \overline{\FF}(0\beta) \cup_{\NN} \TT
\]
is a smooth compact manifold with boundary $\partial\BB=L$. Furthermore, the fiber product map $\widetilde{\ev}_{s}\colon \NN\to L$ extends to $\TT$.

\subsection{Stable trivialization and first homology of the bounding manifold}
In \cite[Section 7]{EkholmSmith} it is shown that the tangent bundle $T\BB$ of the manifold $\BB$ is stably trivializable, and thus the connected component with boundary $L$ is parallelizable. The proof of this fact uses the index bundle over the space of smooth disks in $\C^{4k}$ with boundary in $L$. Using Lemma \ref{Lem:GaussInterpolate} the study of this index bundle reduces to the study of the corresponding index bundle for the surgery on the Whitney sphere and that is then carried out by solving the unperturbed linearized $\bar{\partial}$-equation for disks with boundaries in a suitable skeleton of the loop space of $L \approx_{C^0} \textrm{image}(w)$.

\begin{Lemma}\label{lma:goodfilling}
  The projection to the first factor $L\to S^1$ extends to a map $L \subset \BB \to S^1$.
\end{Lemma}   

\begin{proof}
Let $\mathcal{X}(j\beta)$ denote the space of maps $(D,\partial D)\to (\C^{4k},L)$ with restriction to the boundary in the class $j\beta\in H_{1}(L)$. Then we have $\overline{\FF}(0\beta)\subset \XX(0\beta)$. The space $\XX(0\beta)$ fibers over the loop space of $L\approx S^{1}\times\Sigma$ with contractible fibers and the inclusion of the constant maps into this loop space has a right inverse $\ev_1\colon \XX(0\beta) \to L$ given by evaluation at $1\in\partial D$. The restriction of $\ev_{1}$ to $\overline{\FF}(0\beta)$ gives a retraction back to $L$ which we can compose with the projection $\pi\colon L\to S^1$. On the other piece $\TT$ of $\BB$ we have a different map $\widetilde{\ev}_{s}\colon \TT \to L$, see \eqref{fibermap}, which we can compose with the projection $\pi\colon L\to S^1$. These two maps define on the intersection $\FF(0\beta)\cap \TT=\NN\subset \BB$ two maps
\begin{align*}
  \pi\circ \ev_1,\; \pi\circ\tilde{\ev}_s \ \colon \ \NN \to S^1
\end{align*}
that are homotopic. Indeed, considering $\Phi(\NN)$ as a subset of the configuration space $\XX(0\beta)$, the two maps $\pi\circ\ev_1$ and $\pi\circ\widetilde{\ev}_s$ are the compositions of the natural evaluation map on $\XX^{\ast}(0\beta)$ with two different sections of the $\partial D$-bundle $\XX^{\ast}(0\beta)$ over $\Phi(\NN)$ and the evaluation map along the boundary of any  map in $\XX(0\beta)$ is a contractible loop.  

Since these two maps are homotopic we can use a collar neighborhood $\approx\NN\times [0,1]$ of the boundary in $\overline{\FF}(0\beta)$ to interpolate between them. This then gives an extension of the map $\pi\circ \ev_1\colon \overline{\FF}(0\beta)\to S^{1}$ over the remaining part $\TT$ of $\BB$. The lemma follows.
\end{proof}

\begin{proof}[Proof of Theorem \ref{thm:4k-1}]
  The manifold $S^{1}\times\Sigma$ bounds a stably parallelizable manifold $\BB$. By Lemma \ref{lma:goodfilling} we have a map $b : \BB \to S^1$ extending the projection map $S^{1}\times \Sigma\to S^{1}$. If $p\in S^{1}$ is a regular value then $b^{-1}(p)$ is an almost parallelizable filling of $\Sigma$; the connected component containing $\Sigma$ is therefore parallelizable, hence $[\Sigma]\in bP_{4k}$.
\end{proof}

Note that the argument \emph{inter alia} proves that if $\Sigma$ is a homotopy $4k-1$-sphere which does not belong to $bP_{4k}$, then $S^1 \times \Sigma$ admits no Lagrangian embedding into $\bC^{4k}$ of Maslov index $4k$.

\subsection{Comparing different cotangent bundles}
The results so far show that certain cotangent bundles of homotopy spheres are not symplectomorphic to cotangent bundles of the standard sphere. One can also compare $T^*\Sigma$ and $T^*\Sigma'$ for distinct homotopy spheres, by a trick that we present in this section. The trick is based on two simple observations that we explain first.

Let $M$ be a smooth manifold and let $\pi\colon T^{\ast}M\to M$ denote the projection. If $L\subset T^{\ast}M$ is a Lagrangian submanifold and if $U\subset M$ is an open subset then we say that $L$ is \emph{graphical} over $U$ if there are $1$-forms $\beta_1,\dots,\beta_k$ such that 
\[ 
L\cap \pi^{-1}(U) = \bigcup_{j=1}^{k}\Gamma_{\beta_j},
\] 
where $\Gamma_{\beta}$ is the graph of the $1$-form $\beta$. Furthermore if $(x_1,\dots, x_n)\in V\subset \R^{n}$ are coordinates on $U$ then we say that $L$ is locally constant with respect to these coordinates if 
\[ 
\beta_j=\sum_{i=1}^{n} a_{j}^{i}dx_i,
\] 
are constant $1$-forms for $j=1,\dots,k$. Below we will also use the \emph{norm} $|\beta_j|$ of such a constant $1$-form, which we take to be the norm induced by the standard metric on $\R^{n}$,
\[ 
|\beta_{j}|=\sqrt{(a_{j}^{1})^{2}+\dots+(a_{j}^{n})^{2}}.
\]

\begin{Lemma}\label{lma:graph}
Let $L\subset T^{\ast} M$ be an exact Lagrangian submanifold. Then there is a disk $D_q$ around a point $q\in M$ such that $L$ is graphical over $D_{q}$. Furthermore, if $x=(x_1,\dots, x_n)$ are coordinates on $D_q$ then after a Hamiltonian isotopy we can assume that $L$ is constant with respect to $x$ in some smaller disk $D_{q}'\subset D_{q}$ and that the norms of the distinct constant $1$-forms locally representing $L$ around $q$ are all distinct.  
\end{Lemma}

\begin{proof}
The projection $\pi\colon L\to M$ has a regular value at $q$ exactly when $L$ is graphical in some neighborhood of $q$, so the first statement is a consequence of Sard's lemma. Let $q\in M$ be such a point and fix a coordinate neighborhood $U$ with coordinates $x$ in a ball $B_{1}$ of radius $1$ around $q$. Then each sheet defines a smooth section of the bundle of $1$-forms over $U$, $\tilde\beta_{j}$, $j=1,\dots,k$. By general position, after small deformation we may assume that the norms of all the $1$-forms are distinct at the origin of our coordinate system. Consider now Taylor expansions at the origin of the unique primitives vanishing at the origin for these $1$-forms:
\[ 
f_{j}(x)=\sum_{i=1}^{n} a_{j}^{i}x_{i} \ + \ A_{j}(x), \quad\text{where }A_{j}(x)=\mathcal{O}(|x|^{2})
\]      
%Write $B_{r}\subset B_{1}$ for the ball of radius $r<1$ centered at the origin.
Let $\phi\colon B_{1}\to [0,1]$ be a smooth cut-off function that equals $1$ outside the disk of radius $2\epsilon $ and equals $0$ inside the disk of radius $\epsilon$ and such that $|d\phi|\le 5\epsilon^{-1}$. Consider the Lagrangian over $U$ given by the differentials of the functions
\[ 
g_{j}(x)=\sum_{i=1}^{n} a_{j}^{i}x_{i} \ + \  \phi(x)A_{j}(x)
\]      
near the origin, and which agrees with $L$ outside the cotangent bundle of the disk of radius $2\epsilon$. 
Noting that
\[ 
dg_{j}(x)=\sum_{i=1}^{n} a_{j}^{i} dx_{i} \ + \ A_{j}d\phi+\phi\, dA_j,
\quad\text{where }A_{j}d\phi+\phi\, dA_j=\mathcal{O}(\epsilon),
\] 
it is easy to see that $L$ is Hamiltonian isotopic to the exact Lagrangian obtained by replacing the graphs of $df_j$ over $U$ by the graphs of $dg_{j}$ over $U$ provided $\epsilon>0$ is sufficiently small. This finishes the proof. 
\end{proof}

Our next lemma will be used to change the smooth structure of Lagrangian submanifolds. Let $\beta$ be a constant $1$-form on $\R^{n}$ and consider its graph $\Gamma_{\beta}$ in $T^{\ast}\R^{n}$. Also, let $f$ be the primitive of $\beta$ given by the natural pairing: $f(x)=\langle\beta|x\rangle$. Consider the ball $B_{1}\subset \R^{n}$ of radius $1$ and its boundary $S^{n-1}$. Let $\psi\colon S^{n-1}\to S^{n-1}$ be a diffeomorphism and extend $\psi$ to a diffeomorphism $\psi'$ of a collar neighborhood $S^{n-1}\times (-\epsilon,\epsilon)\to S^{n-1}\times (-\epsilon,\epsilon)$ by 
\[ 
\psi'(x,t)=(\psi(x),t).
\]   
Then the differential of $\psi'$ induces a symplectomorphism 
\[  
\Psi\colon T^{\ast} (S^{n-1}\times (-\epsilon,\epsilon))\to T^{\ast} (S^{n-1}\times (-\epsilon,\epsilon)).
\]
Consider the symplectic manifold $X$ obtained by removing $T^{\ast} B_{1}$ from $T^{\ast}\R^{n}$ and then gluing it back by $\Psi$. Since $\Psi$ takes the $0$-section to the $0$-section, $X$ contains a Lagrangian $L_{0}$ diffeomorphic to $\R^{n}$ obtained by gluing the two $0$-sections. 

Consider next $\Gamma_{\beta}$. Note that $\Gamma_{\beta}$ induces a graphical Lagrangian over a neighborhood of the boundary of $S^{n-1}$ in $B_{1}$. More precisely, this Lagrangian is the graph of the $1$-form $d(f\circ\psi')$. Extending $f\circ\psi'$ to a function $h\colon B_{1}\to\R$ we get a Lagrangian $L_{h}\subset X$ given by $\Gamma_{\beta}$ outside $B_{1}$ and by the graph of $dh$ inside $B_{1}$.   

\begin{Lemma}\label{lma:hcobordism}
With notation as above, if $n>4$ then, for any $\epsilon>0$ there exists a Hamiltonian isotopy of $L_h$ supported in $T^{\ast}B_{2}$ such that $L_{h}\cap L_{0}=\varnothing$. 
\end{Lemma}

\begin{proof}
Consider first the (trivial) case when $\psi$ is isotopic to the identity. Then by isotoping the gluing map we obtain a graphical Lagrangian in $T^{\ast}\R^{n}$ that agrees with $\Gamma_{\beta}$ outside $T^{\ast}B_{1}$. One can use a primitive of the $1$-form given by the fiberwise difference of the two graphs $\Gamma_{h}$ and $\Gamma_{\beta}$ to deform $\Gamma_{h}$ to $\Gamma_{\beta}$. Since $\Gamma_{\beta}\cap\Gamma_{0}=\varnothing$ this finishes the proof.

Consider next the less trivial case when $\psi$ is not isotopic to the identity. Note first that this implies $n\ge 6$ by \cite{KM}. The main difference in this case is that it is not \emph{a priori} clear that there exists a function playing the role of the initial function $f$ above. To see that such a function indeed does exist, consider the smallest cylinder over the plane $\ker(\beta)\subset\R^{n}$ that contains $B_{1+\epsilon}$. Then the gradient of the function $h$ points into the cylinder along the bottom and out of it along the top and is boundary parallel along the rest of the boundary. After small perturbation, the function $h$ has non-degenerate critical points insider the cylinder and by handle theory (the $h$-cobordism theorem) all the critical points can be canceled to straighten the function on the cylinder. This results in a function without critical points, which then enables one to construct the desired Hamiltonian isotopy exactly as before.      
\end{proof}

Using Lemmas \ref{lma:graph} and \ref{lma:hcobordism} we can distinguish cotangent bundles of distinct homotopy spheres neither of which are the standard sphere. More precisely we have the following result, which then implies Theorem \ref{Thm:One} from the Introduction.

\begin{Lemma} \label{Lem:KraghTrick} 
Let $\Sigma$ and $\Sigma'$ be  homotopy spheres of dimension $n>4$ and suppose 
 $\Sigma$ admits a Lagrangian embedding in $T^*\Sigma'$.  A choice of orientation on $\Sigma'$ induces one on $\Sigma$ such that there is a Lagrangian embedding $\Sigma'' \to T^*S^n$, where $[\Sigma''] = [\Sigma]-[\Sigma'] \in \Theta_n$.
\end{Lemma}

\begin{proof} Let $\pi: T^*\Sigma' \rightarrow \Sigma'$ denote the projection. The main result of \cite{FSS,Nadler} implies that $\pi|_{\Sigma}$ has degree $\pm 1$.  Fix an orientation on $\Sigma'$, and then take the induced orientation on $\Sigma$ so  $\pi|_{\Sigma}$ has degree one. By Lemma \ref{lma:graph} we may assume that there is a disk neighborhood $U$ around a point $q\in\Sigma'$ with coordinates $x\in B_{1} \subset U$ where $\Sigma$ is given by a collection of $2k+1$ constant $1$-forms $\beta_1,\dots,\beta_{2k+1}$ of pairwise distinct norms, $|\beta_{j}|\ne |\beta_{l}|$ if $j\ne l$. Here $\pi$ is orientation-preserving on $k+1$ disks and orientation-reversing on $k$ disks in $\Sigma$.

We now apply the cut and paste operation of Lemma \ref{lma:hcobordism} to the disk $B_{1/2}\subset \Sigma'$ using the inverse $\psi\colon S^{n-1}\to S^{n-1}$ of the patching diffeomorphism of $\Sigma'$. The resulting symplectic manifold $X$ is then symplectomorphic to the cotangent bundle of the standard sphere $T^{\ast} S^{n}$ and the canonical Lagrangian submanifold $L_0$ of Lemma \ref{lma:hcobordism} is simply the $0$-section. Again by Lemma \ref{lma:hcobordism} there is a non-zero exact $1$-form $\gamma_1$ over $B_{1}$ that extends $\beta_1$ (which is constant in a neighborhood of $\partial B_{1}$). For $j>1$, pick an orthonormal linear transformation $R_{j}\colon\R^{n}\to\R^{n}$ such that 
\[ 
\beta_{j}= \tfrac{|\beta_{j}|}{|\beta_{1}|} \beta_{1}\circ R_{j}.
\] 
Then
\[ 
\gamma_{j}= \tfrac{|\beta_{j}|}{|\beta_{1}|}\gamma_{1}\circ R_{j},
\]
is an exact non-zero $1$-form that extends $\beta_{j}$. 

We show next that the graphs of $\gamma_{k}$ and $\gamma_{l}$ do not intersect if $k\ne l$. Indeed, such an intersection would correspond to 
\begin{equation}\label{eq:1form=0}
\gamma_{k}(x)-\gamma_{l}(x)=|\beta_{1}|^{-1}\gamma_{1}(x)\circ(|\beta_{k}|R_k-|\beta_{l}|R_{l})=0
\end{equation}
at some $x\in B_{1}$. However $\gamma_{1}(x)\ne 0$ for all $x$ and hence \eqref{eq:1form=0} implies that 
\[ 
(|\beta_{k}|R_k-|\beta_{l}|R_{l})v=0
\]   
for some $v$ with $|v|=1$. That however contradicts $R_{j}$ and $R_{k}$ being orthonormal and $|\beta_k|\ne |\beta_{l}|$.

To finish the proof we take $\Sigma''\subset X$ to be given by $\Sigma$ outside $B_{1}$ and by the graphs $\Gamma_{\gamma_j}$ inside $B_{1}$, with the orientation inherited from $\Sigma$. Then $\Sigma''$ is an embedded Lagrangian submanifold, and the oriented diffeomorphism class of $\Sigma''$ is $[\Sigma]-[\Sigma']$. 
\end{proof}

\section{Further observations}

\subsection{Lagrangians in projective space} Abouzaid \cite{Abouzaid} showed that there are constraints going beyond the homotopy type for exact Lagrangians in cotangent bundles.  It is also natural to look for such constraints on monotone Lagrangians in closed symplectic manifolds.  A first result of that flavour, Theorem \ref{Thm:Two}, follows fairly directly from Theorem \ref{Thm:One}.

\begin{Theorem}
Let $n=4k-1$ and let $\Sigma$ be a homotopy $n$-sphere.  If there is a Lagrangian embedding $\bR\bP^n \# \Sigma \hookrightarrow \bC\bP^n$ then $[\Sigma \# \Sigma] \in bP_{4k}$.
\end{Theorem} 

\begin{proof}
Consider first the case $n=8k-1$. 
Let $L=\bR\bP^n \# \Sigma$, with $[\Sigma]\in\Theta_{8k-1}$.  Since $L \simeq \bR\bP^n$ has 2-torsion $H_1$, the embedding $L \hookrightarrow \bC\bP^n$ is necessarily monotone, of Maslov index either $n+1$ or $2n+2$.
%new
The group $HF^*(L,L)$ is a module over $QH^*(\bC\bP^n)$ (see e.g. \cite{Seidel:LFDAS} or \cite{BiranCornea}), which has an invertible element of degree 2. It follows that $HF^*(L,L)$ is two-periodic.  This implies that the Maslov index of $L$ is $n+1$ and that the restriction of the hyperplane class from $H^2(\bC \bP^n;\bZ_2)$ to $H^2(L;\bZ_2)$ is non-trivial.
%newend

  Let $\hat{L}$ denote the restriction to $L$ of the unit sphere bundle of the line bundle $\mathcal{O}(1) \rightarrow \bC\bP^n$, which is a circle bundle over $L$.  Via Biran's circle bundle construction \cite{Biran}, one obtains a monotone Lagrangian embedding of $\hat{L}$ into $\bC\bP^{n+1} \backslash \bC\bP^n = B^{2n+2} \subset \bC^{n+1}$, again of Maslov index $n+1$. General results of Damian \cite{Damian} imply that any monotone Lagrangian in $\bC^{n+1}$ of Maslov index $n+1$ is diffeomorphic to a fibration over $S^1$ with fibre a homotopy $n$-sphere.  Using the fact that $\hat{L}$ is also circle-fibred, or by direct inspection (pull back the circle bundle to the double cover, where it becomes trivial), one checks that $\hat{L}$ is diffeomorphic to $S^1 \times (\Sigma \# \Sigma)$.  

Since $8 | (n+1)$, the Lagrangian Gauss map is necessarily homotopic to that of the surgery on the Whitney immersion, which implies that $\hat{L}$  bounds a parallelisable manifold $\mathcal{B}$, cf. \cite[Corollary 1.3]{EkholmSmith} (the dimension constraint arises via $\pi_{8k-1}(U/O) = 0$).  Arguing as in Theorem \ref{thm:4k-1}, one can assume that $H_1(\hat{L}) \rightarrow H_1(\mathcal{B})$ is (split) injective, which then implies that $[\Sigma \# \Sigma] \in bP_{8k}$ by taking the regular value of a circle-valued function as before.

Finally, when $n=8k-5$, the same argument carries over except that there is one additional step in the argument with the Gauss map.  Namely, $\pi_{8k-5}(U/O) = \bZ/2$. The homotopy of Lagrangian Gauss maps involves an extension problem over an $n$-cell in $\hat{L}$ where the relevant map is lifted from a corresponding map on a top cell of $\bR\bP^n \# \Sigma$, which means the obstruction lies in the set of 2-divisible elements of $\pi_{8k-5}(U/O)$, hence again vanishes.  This completes the proof.
\end{proof}

\subsection{Coverings}
The \emph{inertia group} of an $n$-manifold  $M$ is the subgroup $I(M)\leq \Theta_n$ of oriented homotopy $n$-spheres $\Sigma$ for which the connect sum $M\#\Sigma$ is diffeomorphic  to $M$.  If $M$ is not orientable, then (by moving the point at which one connect sums around an orientation-reversing loop) one sees that $2\Theta_n \leq I(M)$.  For orientable manifolds, $I(M)$ is often much smaller.  Browder proved \cite{Browder} that for a manifold homotopy equivalent to $\bR\bP^{4k-1}$, $I(M)$ has trivial intersection with $bP_{4k}$; it does not seem to have been computed completely in the literature.  

\begin{Proposition}
Let $\Sigma$ be a homotopy $(4k-1)$-sphere.
If $[\Sigma] \in \Theta_{4k-1}$ satisfies $2[\Sigma] \not \in bP_{4k}$, then $T^*(\bR\bP^{4k-1})$ and $T^*(\bR\bP^{4k-1}\#\Sigma)$ are not symplectomorphic.
\end{Proposition}

\begin{proof} 
This is immediate from Theorem \ref{thm:4k-1} on taking double covers. 
\end{proof}

Note that one sees \emph{a posteriori} that the zero-sections are not diffeomorphic, hence $[\Sigma] \not \in I(\bR\bP^{4k-1})$.  For instance, $|\Theta_{23} / bP_{24}| = 48$, and the result can be applied to distinguish the cotangent bundle of a fake $\bR\bP^{23}$. More generally, when a finite group $G$ acts freely on $S^{2k-1}$, Theorem \ref{Thm:One} shows $T^*(S^{2k-1}/G)$ and $T^*((S^{2k-1}/G) \# \Sigma)$ are symplectically inequivalent whenever $|G|\cdot [\Sigma] \not\in bP_{2k}$. 

\begin{Remark} The proof of Theorem \ref{thm:4k-1} shows that $T^*(S^{4k-1} \times S^1)$ and $T^*(\Sigma \times S^1)$ are symplectically distinct when $[\Sigma] \in \Theta_{4k-1} \backslash bP_{4k}$. However, their universal covers $T^*S^{4k-1} \times \bC$ and $T^*\Sigma \times \bC$ are subcritical, hence symplectomorphic (apply the $h$-principle for isotropic submanifolds of submaximal dimension  to the attaching spheres in a handle decomposition; see \cite{CE} for details). Thus, questions concerning discrete group actions on subcritical manifolds go beyond homotopy theory. 
\end{Remark}

\subsection{Constructing Lagrangians in plumbings}

The surgery involved in Lemma \ref{Lem:KraghTrick} can be used to construct a number of superficially surprising Lagrangian embeddings. 

Consider coordinates $x+iy\in\R^{n}\oplus i\R^{n}=\C^{n}$. The \emph{model plumbing} is the open neighbourhood
\begin{equation} \label{eqn:model-region}
\mathcal{U} = \left\{x+iy\in\C^{n}\colon |x|\cdot |y| \leq 1\right\} \subset \bC^n
\end{equation}
of the union of transverse Lagrangian subspaces $\bR^n \cup i \bR^n$.  Given manifolds $M_i$ with base-points $m_i$, pick Riemannian metrics $g_i$ on $M_i$ which are locally flat near $m_i$, and isometric embeddings 
 \[
 \eta_i\colon  (B_1, 0) \hookrightarrow (M_i,m_i)
 \]
 of the unit ball $B_1\subset \bR^n$. There are symplectic embeddings 
\[
D\eta_i\colon \, T^*B_1 \hookrightarrow T^*M_i
\]
which take the unit disk bundle for the flat metric on $B_1$ to the unit disk bundle of $M_i$. On the other hand, the inclusions of the unit ball $\iota_0\colon B_1 \hookrightarrow \bR^n$ and $\iota_1\colon B_1 \hookrightarrow i  \bR^n$ induce symplectic inclusions $D\iota_j : T^*B_1 \hookrightarrow \bC^n$.   The \emph{plumbing} $T^*M_0 \#_{\pl} T^*M_1$ is the Liouville manifold obtained by completing a Weinstein domain obtained from the unit disk bundles $D^*M_i$ by identifying points in the image of $D\eta_0$ with points in the image of $D\eta_1$ whenever their $D\iota_j$-images in $\mathcal{U} \subset \bC^n$ co-incide.   This identifies the two copies of the unit disk bundle $D^*B_1 \cong B_1\times B_1$ via the map $(x,y) \mapsto (y,-x)$, which preserves the symplectic form $\sum_{i=1}^{n}dx_{i}\wedge dy_{i}$.

\begin{Lemma} \label{Lem:Order-2-exists}
Let $\Sigma$ be a homotopy $n$-sphere. If $[\Sigma] \in \Theta_n$ has order two,  there is a Lagrangian embedding of the standard sphere $S^n \hookrightarrow T^*\Sigma \#_{\pl} T^*\Sigma$.  
\end{Lemma}

\begin{proof}
Since $\Sigma \cong -\Sigma$ are oriented diffeomorphic, there is a symplectomorphism
\begin{equation} \label{Eqn:ChangeSign}
T^*\Sigma \#_{\pl} T^*\Sigma \ \cong \ T^*\Sigma \#_{\pl} T^*(-\Sigma).
\end{equation}
The Lagrange surgery of the core components, in the second description of \eqref{Eqn:ChangeSign},  is a Lagrangian $\Sigma \# (-\Sigma) = S^n$.  Alternatively, one can 
 begin with $T^*\Sigma \#_{\pl} T^*S^n$ and take the image of the first core factor under the Dehn twist along the second,  to obtain an exact Lagrangian embedding $\iota$ of $\Sigma$ in a homology class which has degree one over the $S^n$-factor.  The result then follows from surgery as in Lemma \ref{Lem:KraghTrick}; after making the embedding $\iota$ graphical over a disk in $S^n$, one reglues over that disk by a diffeomorphism representing $[\Sigma] \in \Theta_n$, to obtain an embedding of $\Sigma \# \Sigma \cong S^n$ into $T^*\Sigma \#_{\pl} T^*\Sigma$.
 \end{proof}

Generalizing the second proof of Lemma \ref{Lem:Order-2-exists} gives the following. For a homotopy $n$-sphere $\Sigma$ we write $\langle \Sigma\rangle \leq \Theta_n$ for the cyclic subgroup of $\Theta_n$ generated by $\Sigma$.

\begin{Lemma}
Let $n>4$ be odd and let $\Sigma$ be a homotopy $n$-sphere. Every homotopy $n$-sphere defining a class in $\langle \Sigma\rangle$ admits a Lagrangian embedding into $T^*\Sigma \#_{\pl} T^*\Sigma$.
\end{Lemma}

\begin{proof}
Let $X=T^*\Sigma \#_{\pl} T^*S^n$. Consider the image $L$ of the first core component $\Sigma$ under the $m$-th power of the Dehn twist along the second core component $S^n$.  Since $n$ is odd, the Dehn twist acts by an infinite order transvection on homology, so this yields an embedded copy of $\Sigma$ in $X$ in the homology class 
\[ 
(1,m)\in H_{n}(X)=H_{n}(\Sigma)\oplus H_{n}(S^{n})\approx \Z\oplus \Z, 
\]
Make this embedding graphical over a disk in $S^n$, with degree $m$ over the base, and then cut out and reglue as in Lemma \ref{Lem:KraghTrick} (strictly, in a modification of that lemma adapted to the case in which the original Lagrangian has higher degree over the base). This yields a copy of 
\[ 
\Sigma \;\#\; \underbrace{\Sigma \# \dots \# \Sigma}_{m}
\]
in the plumbing $T^*\Sigma \#_{\pl} T^*\Sigma$.
\end{proof}

It is interesting to contrast  Lemma \ref{Lem:Order-2-exists} with the following.
Let $L, L' \subset X$ be closed exact Lagrangian submanifolds of a convex exact symplectic manifold $X$.  Motivated by the analogy between Morse and Floer theory, Fukaya formulated in \cite{Fukaya:s-cob} the \emph{symplectic $s$-cobordism conjecture}, asserting  that when Floer cohomology $HF^*(L,L')$  and its Whitehead torsion both vanish, there should be a Hamiltonian symplectomorphism $\phi\colon X \rightarrow X$ with $\phi(L) \cap L' = \varnothing$, i.e.~displacing $L$ from $L'$.  

\begin{Lemma} \label{lem:s-cobordism}
Let $n>4$ be odd and let $\Sigma$ be a homotopy $n$-sphere with $[\Sigma] \not\in bP_{n+1}$. 
The symplectic s-cobordism conjecture implies that there is no Lagrangian embedding $\Sigma \hookrightarrow T^*S^n \#_{\pl} T^*S^n$.
\end{Lemma}

\begin{proof}
The plumbing $T^{\ast}S^{n}\#_{\pl} T^{\ast}S^{n}$ is symplectomorphic to the complex $n$-dimensional $A_{2}^{n}$-Milnor fiber, i.e.~the affine hypersurface $X\subset\C^{n+1}$ with complex coordinates $(z_1,\dots, z_n,t)$ given by the equation
\[ 
\sum_{j=1}^n z_j^2 + t^3 = 0.
\]
The space $X$ retracts to a chain of two Lagrangian $n$-spheres $L_a$ and $L_b$ that intersects transversely in a point $\{p\}$ (these correspond to the two zero-sections in the plumbing). The compactly supported symplectomorphisms of $X$ that are obtained from Dehn twists in $L_a$ and $L_b$ give a copy of the braid group $Br_{3}$ on three strands inside the group of symplectic isotopy classes of such symplectormorphisms, $\pi_0(\mathrm{Symp}_{\mathrm{ct}}(X))$.

Assume that there exists a Lagrangian embedding of a homotopy $n$-sphere $\Sigma$ into $X$.
By applying symplectomorphisms in $Br_{3}$ one can assume that $\Sigma$ is homologous, and even isomorphic in the Fukaya category $\scrF(T^{\ast}S^{n}\#_{\pl} T^{\ast}S^{n})$, to the first sphere $L_a$ in the chain, see \cite{AbSm}.  

Then, if there is some point $x\in L_b\backslash \{p\}$ such that $\Sigma\,\cap \,T_{x}^{\ast} L_{b}=\varnothing$, where $T_x^*L_b \subset X$ is the cotangent fiber of $L_{b}$ at $x$, then in fact $\Sigma$ lies inside an open subset of $X$ of the form $T^*L_a \#_{\pl} T^*\bR^n$, where $\bR^n = L_b \backslash \{x\}$. This subset is just the boundary connect sum of $T^{\ast}L$ with a ball. Boundary connect sum with a ball does not change symplectomorphism type and we would find a Lagrangian embedding $\Sigma \hookrightarrow T^*L_a = T^*S^n$. That contradicts (the proof of) Theorem \ref{Thm:One} whenever $[\Sigma] \not\in bP_{n+1}$ and proves that in fact $\Sigma\cap T_{x}^{\ast} L_{b}\ne\varnothing$ for every $x\in L_{b}$.  

We next note that $HF(\Sigma, T_x^*L_b) = 0$ since $\Sigma$ is isomorphic in the Fukaya category to $L_a$. Since we are dealing with simply-connected Lagrangians in a simply-connected exact manifold, the Floer complex is defined over $\bZ$; the Whitehead group $Wh(\{1\}) = 0$, and so the Floer complex furthermore has trivial Whitehead torsion.
Finally, view the $A_3^{n}$-Milnor fiber $M$ as obtained from the $A_2^{n}$-Milnor fiber $X$ above by attaching an $n$-dimensional handle $(H,H_0)$, where $H_0$ denotes the Lagrangian core disk, to the Legendrian unknot which is the boundary of the cotangent fiber $T_x^*L_b$.  The space $M$ then contains the closed Lagrangian submanifolds $L = \Sigma$ and $L' = T_{x}^{\ast}L_b\cup H_0 \cong S^n$ (the newly created third core component). Then $HF^{\ast}(L,L')\approx HF(\Sigma, T_x^*L_b) = 0$ and the Whitehead torsion still vanishes. Thus $L$ and $L'$ would violate the hypotheses of the $s$-cobordism conjecture, and hence the original embedding of $\Sigma$ could not exist.   
\end{proof}

\subsection{Constraining Lagrangians in boundary sums} 
In place of the plumbing $T^*\Sigma \#_{\pl} T^*\Sigma'$ of two cotangent bundles, one can also consider the boundary connect sum $M = T^*\Sigma \#_{\partial} T^*\Sigma'$, obtained by joining the Stein manifolds by a (subcritical) Stein one-handle (applied to the closed unit disk cotangent bundles, this precisely takes connect sum at the contact boundary).  Every closed exact spin Lagrangian in $M$ defines an object in the compact Fukaya category $\scrF(M)$, in particular there are two distinguished objects which are just the Lagrangians $\Sigma, \Sigma'$; the manifold $M$ retracts to an isotropic skeleton obtained by joining these two core Lagrangians by an arc.

\begin{Lemma} \label{Lem:All_cores}
Every closed exact graded spin Lagrangian in $M$ is isomorphic to a shift of one of the core components  $\Sigma$ or $\Sigma'$ in $\scrF(M)$.
\end{Lemma}

\begin{proof}
This follows from the same (somewhat roundabout) argument used to study the Fukaya category of the plumbing in \cite{AbSm}.  More precisely, as in \cite{FSS}, by complexifying a Morse function on $\Sigma \amalg \Sigma'$ one obtains an embedding $M \hookrightarrow E$ of a relatively compact open neighborhood of the skeleton of $M$ as a Liouville subdomain of an affine variety $E$ which is equipped with a Lefschetz fibration $p\colon E \rightarrow \bC$.  Combining Seidel's theorem that the Lefschetz thimbles generate the Fukaya category of the Lefschetz fibration \cite{Seidel:FCPLT}, and the acceleration functor of \cite{AbouzaidSeidel}, one sees \cite[Proposition 3.6]{AbSm} that $\scrF(M)$ lies inside the subcategory of the wrapped Fukaya category $\scrW(E)$ generated by the Lefschetz thimbles.  Now apply the Viterbo-type restriction functor of \cite{AbouzaidSeidel} from $\scrW(E)$ to $\scrW(M)$; each Lefschetz thimble is either disjoint from $M$ or co-incides with a disk cotangent fibre inside $M$, which implies that $\scrF(M)$ is generated by the two cotangent fibres.   

Finally, since the relevant wrapped Floer complexes are graded in non-positive degrees, the category of compact Lagrangians $\scrF(M)$ is generated by the core components, i.e.~the simple modules over the cotangent fibers, compare to \cite[Lemmas 4.6 and 4.7]{AbSm}.   Since these two simple modules are Floer-theoretically disjoint, the only twisted complexes which are irreducible (have self-Floer-cohomology of rank $1$ in degree $0$) are the core components themselves, up to shift.
\end{proof}

\begin{Remark} There are at least two other routes to proving the previous result, which do not invoke the auxiliary Lefschetz fibration $E$ and which may give at least as much intuition as the actual proof above. 
 \cite{Abouzaid:cotangent_fibre_generates} shows that the open-closed map 
 \[  
HH_*(HW^*(T_x^*,T_x^*)) \rightarrow SH^*(T^*S^n),
\]
hits the unit, implying that the wrapped category of the cotangent bundle is split-generated by by the cotangent fiber by the criterion of \cite{Abouzaid:generation}. For boundary sums, there is an isomorphism 
\begin{equation} \label{Eqn:Split}
SH^*(M_1 \#_{\partial} M_2) \cong SH^*(M_1) \oplus SH^*(M_2)
\end{equation}
(see \cite{Cieliebak:chord, McLean} for the isomorphism of \eqref{Eqn:Split} as vector spaces respectively rings). 
If the open-closed map was compatible with the restriction functor, setting $\scrW_0(M) \subset \scrW(M)$ for the subcategory generated by the two cotangent fibres,  one could immediately conclude that the map 
\begin{equation} \label{Eq:boundary_sum_hit_unit}
HH_*(\scrW_0(M)) \rightarrow SH^*(M)
\end{equation}
hits the unit, implying $\scrW_0(M)$ split-generates $\scrW(M)$ and in particular $\scrF(M)$.  However, compatibility of the open-closed map and the restriction functor does not yet appear in the literature.  

In another direction, one can verify that \eqref{Eq:boundary_sum_hit_unit} hits the unit using the technology of \cite{BEE}.  For consistency with \cite{BEE}, we shift grading so the unit in $SH^*$ now lies in degree $n$.  We can represent $M$ as the result of attaching Lagrangian handles to two distant standard Legendrian spheres $\Lambda_1$ and $\Lambda_{2}$ in the boundary of the ball (whence there are effectively no Reeb chord connecting $\Lambda_{1}$ to $\Lambda_{2}$, meaning none up to any desired value of action). We represent that ball as an ellipsoid, see \cite[Section 7.1]{BEE} which means that it has effectively only one Reeb orbit $\gamma$ of Conley-Zehnder index $|\gamma|=n+1$. Continuing to follow \cite[Section 7.1]{BEE}, the Reeb orbits after surgery of Conley-Zehnder index $<2n-2$ correspond to the one-letter words of Reeb chords $a_1$ and $a_2$, and both have Conley-Zehnder index $n-1$. Denote these orbits $\{a_j\}$. We next consider the symplectic cohomology complex corresponding to a small perturbation of a time independent Hamiltonian with one minimum $m$ and two critical points $p_{\pm}$ of index $n$, one in each handle.  Considering the Morsification of the critical manifold, a disjoint union of circles, parameterizing a choice of start-point on the Reeb orbits, one sees each Reeb orbit $\gamma$ gives rise to two generators $\check\gamma$ and $\hat{\gamma}$ for the $SC^*$-complex. We thus have the following subcomplex near the degree $n$ of the unit:
\begin{itemize}
\item In degree $n$ we have the generators $m$, $\{\hat a_1\}$, and $\{\hat a_2\}$ and the differential acts trivially. (It is clear the differential is trivial on $m$ and on $\{\hat a_j\}$ it is simply the Morse differential of the circle, which vanishes; see \cite{BourgeoisOancea}.)
\item In degree $n+1$ there is the generator $\check{\gamma}$ and the differential of $\check{\gamma}$ is $m-\{\hat a_1\}-\{\hat a_2\}$. To see this, note that \cite{BEE} gives an isomorphism $\Phi\colon SH^*(M)\to LH^{\mathrm{Ho}}(\Lambda_1\cup\Lambda_2)$. Here the differential on $LH^{\mathrm{Ho}}(\Lambda_1\cup\Lambda_2)$ is trivial, and in degree $n$ the complex is generated by $\hat a_j$, $j=1,2$. Further \cite[Section 5]{BEEproduct} shows that $\hat a_1+\hat a_2$ is the image of the unit under the map $\Phi$ and as we shall see below the preimage of $\hat a_j$ under $\Phi$ is the orbit $\{\hat a_{j}\}$. On the other hand $m$ also represents the unit because the natural map $H^*(M) \rightarrow SH^{*+n}(M)$ is unital. Thus the two cycles must be homologous.
\end{itemize} 
We consider the map $\Phi$ in more detail. The cotangent fibers appear in the handle model of $M$ as the co-core spheres $C_1$ and $C_2$, with Legendrian boundary spheres $\Gamma_1$ and $\Gamma_2$, for the handles attached to $\Lambda_1$ and $\Lambda_2$, respectively. It follows from \cite[Section 6.1]{BEE} that the Reeb chords of $\Gamma_j$ are in natural one to one correspondence with words of Reeb chords of $\Lambda_j$. As explained in \cite[Section 6.3]{BEE}, the isomorphism $\Phi$ is obtained from moduli spaces of strips with two Floer corners at the intersection of the core and co-core disks.  One considers strips which connect the chord on the co-core sphere corresponding to the chord $a_j$, to $a_j$ itself. Self-gluing of such a strip gives a 1-dimensional moduli space of holomorphic annuli. For action reasons the only possible splitting at the other end of the moduli space is a two-story building, with upper level a disk with positive boundary puncture at the chord of $\Gamma_j$ corresponding to $a_j$ and negative interior puncture at $\{a_j\}$, and with lower level having a  positive puncture at $\{a_j\}$ and negative puncture at $a_j$. The upper disk contributes to the open-closed string map, whilst the lower disk defines the surgery isomorphism map. It follows that $[\hat a_1]+[\hat a_2]$, which represents the unit, is indeed in the image of the open-closed map.  

Finally, in the spirit of the previous paragraph, we note that there is a still more direct proof (again not yet completely in the literature) which involves an ``upside down" surgery map. For this, one uses the approach in \cite{BEEproduct} to deal with holomorphic curves with several positive punctures, and directly defines an isomorphism from a mapping cone 
\[
\mathrm{Cone}\left\{ HH_*(\scrW_0(M)) \rightarrow SH^*(M) \right\} \longrightarrow SH^*(B^{2n})
\]
of the open-closed map (from Hochschild homology of the wrapped homologies of the co-core spheres to the symplectic cohomology after surgery)  to the (acyclic) symplectic homology of the subcritical manifold before surgery. This again establishes Abouzaid's generation criterion. 
\end{Remark}

\begin{Lemma}
Let $n>4$ be odd and let $\Sigma$ be a homotopy $n$-sphere. If $[\Sigma] \not\in bP_{n+1}$ then the boundary connect sums 
\[
T^*S^n \#_{\partial} T^*S^n, \ T^*\Sigma \#_{\partial} T^*S^n \text{ and } \ T^*\Sigma \#_{\partial} T^*\Sigma
\]
are pairwise non-symplectomorphic. 
\end{Lemma}

\begin{proof}
We claim the manifolds are distinguished by the set of homology classes which can be represented by an embedded Lagrangian submanifold diffeomorphic to $\Sigma$. Consider for definiteness the first space $M = T^*S^n \#_{\partial} T^*S^n$.  This is the completion of a Stein neighbourhood of the isotropic complex obtained from joining two disjoint Lagrangian $n$-spheres by an arc.   Abouzaid's proof of the non-existence of homotopy spheres in $bP_{4k+2}$ as embedded Lagrangians in $T^*S^{4k+1}$ begins with the existence of a Lagrangian embedding of $S^{4k+1}$ into $\bC^{2k+1}\times \bC\bP^{2k}$, via the graph of the Hopf map. Taking a pair of disjoint such spheres (via translation in the $\bC^{2k+1}$-factor), if $\Sigma$ with $[\Sigma] \in \Theta_{4k+1}\backslash bP_{4k+2}$ admits any Lagrangian embedding into $M$, it admits a Lagrangian embedding into $\bC^{2k+1}\times \bC\bP^{2k}$.  The Maslov class of that embedding is equal to that of either copy of the Hopf graph $S^{4k+1}$, since by Lemma \ref{Lem:All_cores} $\Sigma$ is necessarily homologous to one of the core components. This is enough to run Abouzaid's argument and obtain a contradiction.  For other odd values of $n$ one analogously adapts the arguments of \cite{EkholmSmith}, joining two disjoint Lagrangian embeddings of $S^{1}\times S^{4k-1}$ by an isotropic annulus corresponding to connecting by arcs fiberwise over $S^1$.

Now suppose there is an embedding of $\Sigma\cong L$ into $T^*\Sigma \#_{\partial} T^*S^n$ in the homology class of the second component.  By a version of Lemma \ref{Lem:KraghTrick}, we can make $L$ graphical over a disk in the core component $\Sigma$ (over which $L$ has degree zero by hypothesis), and cut and paste  to obtain an embedding of $\Sigma$ into $T^*S^n \#_{\partial} T^*S^n$, violating the previous paragraph. The cut-and-paste argument similarly reduces the boundary sum of two cotangent bundles of exotic spheres to the corresponding boundary sum of cotangent bundles of standard spheres. The result follows.
\end{proof}

\subsection{Unit cotangent bundles}
The starting point for this paper is the question of how the symplectic geometry of the cotangent bundle of a smooth manifold depends on the smooth structure of the manifold. An analogous contact geometric question is how the contact geometry of the unit cotangent bundle of a manifold depends on the smooth structure. We are not quite able to prove the counterparts of the symplectic theorems above for exotic spheres. Previous results however do give some information that we discuss next. 

Let $M$ be a smooth manifold and let $U^{\ast}M$ denote its spherical cotangent bundle with its standard contact structure given by restriction of the Liouville form $p\,dq$. Note that the fiber $F\subset U^{\ast}M$ is a Legendrian sphere. 

\begin{Theorem}\label{t:unit+fiber}
Let $M_0$ and $M_1$ be smooth manifolds.
If $M_0$ does not admit an exact Lagrangian embedding into $T^*M_1$ then the pairs $(U^{\ast}M_{0},F)$ and $(U^{\ast} M_1,F)$ are not contactomorphic.
\end{Theorem}    

\begin{proof}
Assume that $U^{\ast}M_0$ and $U^{\ast}M_1$ are contactomorphic via a contactomorphism that takes the fiber to the fiber. Let $f$ be a Morse function on $M_0$ and consider the graph of $\lambda df$ for  $\lambda>0$. By taking $\lambda$ large we find that $M_0$ admits a Lagrangian embedding that intersects the unit cotangent bundle along a finite collection of fiber disks. Removing these fiber disks we find that there is a finite subset $X \subset M_0$, and an exact Lagrangian embedding of $M_{0}-X$ into the symplectization of $U^{\ast}M_0$ which agrees with a union of cones on fiber spheres at the negative end. After isotopy we may assume that all the fiber spheres lie in an arbitrarily small neighborhood of a fixed fiber sphere. Applying the contactomorphism we then get an analogous embedding of $M_0-X$ into the symplectization of $U^{\ast}M_1$. Filling the fibers at the negative end then gives an exact Lagrangian embedding of $M_0$ into $T^{\ast}M_1$. The result follows.   
\end{proof}

\begin{Corollary}
If $\Sigma_0$ and $\Sigma_{1}$ are homotopy $n$-spheres, $n$ odd, and if $[\Sigma_0]\ne[\pm \Sigma_{1}]\in\theta_n/bP_{n+1}$ then the pairs $(U^{\ast}\Sigma_0,F)$ and $(U^{\ast}\Sigma_1,F)$ are not contactomorphic. \qed
\end{Corollary}

%\begin{proof}
%Follows from Theorems \ref{Thm:One} and \ref{t:unit+fiber}.
%\end{proof}

\subsection{Legendrian spheres} 
We conclude by making a few observations concerning the implications of our results for Legendrian knot theory. Consider the standard handle presentation for the cotangent bundle $T^{\ast}S^{n}$. There is a $0$-handle, which is the ball, with a standard Legendrian sphere in the boundary along which the $n$-handle is attached. The surgery presentation of the cotangent bundle of an exotic sphere is the same except that the parameterization of the knot (that describes how the handle is attached) is different. 

Let $\Lambda$ denote the standard Legendrian unknot of dimension $n$,  in the boundary of the standard ball. Consider two parametrisations $\phi_0$ and $\phi_1$ of $\Lambda$ which differ by a  diffeomorphism $\phi\colon \Lambda\to\Lambda$.  Viewed as a gluing map for two hemispheres, any such $\phi$ defines a class in $\Theta_{n+1}$.  The following result is a non-existence theorem for an isotopy of maps, rather than of submanifolds.

\begin{Corollary}\label{c:paramisotopy}
Suppose $n$ is even. (i)  If $\phi$  is non-trivial in $\Theta_{n+1}/bP_{n+2}$ then the $\phi_i$  are not isotopic through parametrized Legendrian embeddings. 
(ii) The parameterization of the standard unknot induced from any Lagrangian disk filling defines the trivial element in $\Theta_{n+1}/bP_{n+2}$ (when used as the gluing map for a homotopy $(n+1)$-sphere).
\end{Corollary}

\begin{proof}
An isotopy connecting the two embeddings would give a symplectomorphism between the cotangent bundles, viewed as having been constructed by handle attachment, in contradiction to Theorem \ref{Thm:One}. A Lagrangian disk which induces a non-trivial parameterization would give an exact Lagrangian embedding of an exotic sphere after handle attachement.
\end{proof}

Theorem \ref{t:nearbyparam} on re-parameterized 0-sections of cotangent bundles follows from Corollary \ref{c:paramisotopy}.
\begin{proof}[Proof of Theorem \ref{t:nearbyparam}]
A Hamiltonian isotopy in $T^{\ast} S^{n}$ lifts to a contact isotopy of the $1$-jet space $J(S^n)  =T^{\ast}S^{n}\times \R$.  Suppose for contradiction that, in the setting of Theorem \ref{t:nearbyparam}, there is a Hamiltonian isotopy connecting the maps $\iota$ and $\iota\circ\phi$.  After scaling the $\R$-coordinate and the fiber coordinates in $T^{\ast}S^{n}$ we may assume that the induced Legendrian isotopy from the $0$-section $\Lambda = S^n \subset J(S^n)$ parameterized by $\iota$ to that parameterized by $\iota\circ\phi$ lies in an arbitrarily small neighborhood of $\Lambda$.   Since any Legendrian sphere has a neighborhood contactomorphic to a neighborhood of the 0-section in the 1-jet space,  such an isotopy contradicts Corollary \ref{c:paramisotopy}. 
\end{proof}

One may consider analogues of Corollary \ref{c:paramisotopy} for the fiber sphere $F\subset U^{\ast}\Sigma$, where $\Sigma$ is an exotic sphere. View $\Sigma$ as the union of two disks $D_+$ and $D_-$ glued by a diffeomorphism $\phi$ and consider the isotopy which flows $F$ over $D_{+}$ radially, then shrinks it radially over $D_{-}$, and finally moves it back to its initial position along a path in the base. It is easy to see that this isotopy changes the original parameterization by $\phi$. In general, wrapping many times around the sphere we can change it by any multiple of $\phi$. The trace of this isotopy is a Lagrangian cylinder that gives a cobordism from the fiber to the fiber, and  by filling we find that there are Lagrangian disks bounding the fiber in the cotangent bundle of an exotic sphere that do not induce  the standard parameterization. It is an open question whether there are such isotopies changing the parameterization of the fiber in the unit cotangent bundle of the standard sphere (this is closely related to the question raised in Lemma \ref{lem:s-cobordism} as to whether an exotic sphere admits a Lagrangian embedding into the $A_2$-plumbing  of cotangent bundles of the standard sphere).

\bibliographystyle{alpha}

%%%
\end{document}